\documentclass[12pt,a4paper]{article}
\usepackage[cp1251]{inputenc}
\usepackage[english]{babel}
\usepackage{color}
\usepackage[colorlinks=true,linkcolor=blue,filecolor=blue,citecolor=blue,urlcolor=blue]{hyperref}
\usepackage{amsmath,amssymb,amsthm}
\usepackage{graphicx}
\usepackage{comment}
\usepackage{enumitem}
\usepackage{subfigure,multirow}
\usepackage{amsthm}
\usepackage{thmtools}
\declaretheoremstyle[notefont=\bfseries,notebraces={}{},%
   headpunct={},postheadspace=1em,bodyfont=\it]{mystyle}
\declaretheorem[style=mystyle,numbered=no,name=Theorem]{thm-hand}

\usepackage{geometry}
\DeclareMathOperator{\Exp}{\mathrm{Exp}}
\DeclareMathOperator{\sspan}{\mathrm{span}}

\DeclareMathOperator{\diam}{\mathrm{diam}}

\newcommand{\tmax}{t_{\mathrm{max}}}
\newcommand{\tconj}{t_{\mathrm{conj}}}
\newcommand{\tcut}{t_{\mathrm{cut}}}

\newcommand{\su}{\mathfrak{su}}

\newcommand{\U}{\mathrm{U}}

\newcommand{\R}{\mathbb{R}}
\newcommand{\C}{\mathbb{C}}

\newcommand{\Z}{\mathbb{Z}}

\newcommand{\id}{\mathrm{id}}
\newcommand{\SU}{\mathrm{SU}}
\newcommand{\SL}{\mathrm{SL}}

\newcommand{\SO}{\mathrm{SO}}

\DeclareMathOperator{\closure}{\mathrm{cl}}

\newcommand{\Cut}{\mathrm{Cut}}
\newcommand{\Sym}{\mathrm{Sym}}

\newcommand{\tauconj}{\tau_{\mathrm{conj}}}
\newcommand{\taul}{\tau_{\ell}}

\theoremstyle{definition}
\newtheorem{definition}{Definition}
\newtheorem{remark}{Remark}

\theoremstyle{plain}

\newtheorem{lemma}{Lemma}
\newtheorem{theorem}{Theorem}
\newtheorem{proposition}{Proposition}

\newenvironment{enumerate*}%
  {\begin{enumerate}%
    \setlength{\itemsep}{1pt}%
    \setlength{\parskip}{1pt}}%
  {\end{enumerate}}

\title{Cut loci and diameters of the Berger lens spaces\footnote{This work was supported by the Russian Science Foundation under grant no.~25-21-00681, \href{https://rscf.ru/project/25-21-00681/}{https://rscf.ru/project/25-21-00681/} and performed in Ailamazyan Program Systems Institute of Russian Academy of Sciences.}}

\author{
A.\,V.~Podobryaev \\ A.\,K.~Ailamazyan Program Systems
Institute of RAS \\ \tt{alex@alex.botik.ru} \\
}

\date{}

\begin{document}

\maketitle

\begin{abstract}
In this paper, we study Riemannian metrics on the three-dimensional lens spaces that are deformations of the standard Riemannian metric along the fibers of the Hopf fibration. In other words, these metrics are axisymmetric. There is a one-parametric family of such metrics. This family tends to an axisymmetric sub-Riemannian metric. We find the cut loci and the cut times using methods from geometric control theory. It turns out that the cut loci and the cut times converge to the cut locus and the cut time for the sub-Riemannian structure, which has already been studied. Moreover, we get some lower bounds for the diameter of these Riemannian metrics. These bounds coincide with the exact values of diameters for the lens spaces $L(p;1)$.


\end{abstract}

\section{\label{sec-introduction}Introduction}

This paper is a sequel of papers~\cite{podobryaev-sachkov-so3,podobryaev-berger-sphere-diameter} where we found cut loci and diameters of left-invariant axisymmetric Riemannian metrics on the Lie groups $\SU_2$ and $\SO_3$. Note that the three-dimensional sphere $\SU_2$ with such a metric is a Berger sphere, i.e., a sphere with the standard metric deformed along the fibers of the Hopf fibration. At the same time, the Lie group $\SO_3 = \SU_2 / \Z_2$ is a special case of a lens space. Therefore, it seems natural to study the cut loci of general lens spaces with this kind of metric.

Moreover, our previous works~\cite{podobryaev-sachkov-so3,podobryaev-sachkov-sl2,podobryaev-sachkov-so3-sl2} were motivated by the fact that axisymmetric left-invariant sub-Riemannian structures on the Lie groups $\SL_2$, $\SO_3$, and $\SU_2$
were studied in the works of V.\,N.~Be\-re\-stov\-skii and I.\,A.~Zubareva~\cite{berestovskii-zubareva-sl2,berestovskii-zubareva-so3,berestovskii-zubareva-su2} and
U.~Boscain and F.~Rossi~\cite{boscain-rossi}. We investigated one-parameter families of Riemannian metrics tending to these sub-Riemannian structures.
In paper~\cite{boscain-rossi}, the authors also considered sub-Riemannian structures on lens spaces.
In the present paper, we consider one-parameter families of invariant axisymmetric Riemannian metrics on lens spaces whose limits are the sub-Riemannian metrics studied in~\cite{boscain-rossi}.

Moreover, considering the family of the lens spaces $L(p;q)$ depending on the parameter $p$ allows us to better understand the reasons for the appearance of an additional strata of the cut locus that arise in the cases of $\SU_2$ and $\SO_3$.

Notice that lens spaces $L(p;q)$ are not homogeneous spaces. So, the cut locus depends on an initial point of the lens space. We compute the cut locus with respect to the initial point $o = \Pi(\id)$, where $\id \in \SU_2$ is the identity element and $\Pi : \SU_2 \simeq S^3 \rightarrow L(p;q)$ is the natural projection.
However, the lens spaces $L(p;1)$ and $L(p;-1)$ are homogeneous with respect to the $\SU_2$-action. In these cases our results don't depend on an initial point.

It should be noted that the cut locus and the diameter of the standard Riemannian metric on lens spaces are known from the results of S.~Anisov~\cite{anisov}. The present paper generalizes these results to the case of metrics deformed along the fibers of the Hopf fibration, which we call the Berger metrics.

The paper has the following structure. In Section~\ref{sec-metric} we recall the definition of the lens space and define a Riemannian metric of the Berger type. Then, in Section~\ref{sec-known-res} we give several necessary definitions and some previous results on geodesics equations and the first conjugate time.
In Section~\ref{sec-sym}, we define symmetries of the problem and prove several technical statements which are necessary to compare Maxwell times corresponding to different kinds of symmetries.
Section~\ref{sec-maxwell} is dedicated to the relative location of Maxwell strata. Also we prove here  that the first Maxwell time for symmetries is not greater than the first conjugate time.
For completeness, we provide proofs for all natural $p$, including the cases $p = 1$ and $p = 2$, which were already studied in paper~\cite{podobryaev-sachkov-so3}.
Section~\ref{sec-cutlocus} contains the main result (Theorem~\ref{th-cutlocus}) which describes the cut locus, also there is Theorem~\ref{th-subRiemannian} about the sub-Riemannian structure as a limit case of Riemannian ones. Finally, we find lower bounds for diameters of the Berger lens spaces, see Theorem~\ref{th-diam} in Section~\ref{sec-diam}.

\section{\label{sec-metric}Axisymmetric Riemannian metrics on lens spaces}

Let us recall the notion of a lens space.

\begin{definition}
\label{def-lens-space}
Let $p,q \in \Z \setminus \{0\}$ be coprime numbers. Consider a three dimensional sphere $S^3 = \{(z,w) \in \C^2 \, | \, |z|^2 + |w|^2 = 1\}$ and a $\Z_p$-action on $S^3$ via the formula:
$$
[k] \circ (z,w) = \left(e^{i\frac{2\pi k}{p}} z, e^{i\frac{2\pi kq}{p}} w\right), \qquad \text{where} \qquad [k] \in \Z_p, \ (z,w) \in S^3.
$$
The orbit space $L(p;q) = S^3 / \Z_p$ is called \emph{a lens space}.
We denote by $\Pi : S^3 \rightarrow L(p;q)$ the natural projection.
\end{definition}

We need some model of the lens space $L(p;q)$. Let $z = q_0 + iq_3$, $w = q_1 + iq_2$.

\begin{proposition}[See, for example, \cite{boscain-rossi}, Prop.~2]
\label{prop-lens-space}
Assume that $p > 1$. Consider the subset
$$
L = \left\{ (z,w) \in S^3 \, \Bigm| \, q_1^2 + q_2^2 + \frac{q_3^2}{\sin^2{\frac{\pi}{p}}} \leqslant 1, \ q_0 \geqslant 0 \right\} \subset S^3.
$$
Let us glue the points $(q_1,q_2,q_3) \sim (q_1',q_2',q_3')$ of the boundary $\partial L$ if one of the following conditions holds\emph{:}\\
\emph{(1)} $q_3 > 0$, $q_3' = -q_3$ and $z' = q_1' + iq_2' = e^{i\frac{2\pi q}{p}}z$, where $z = q_1 + iq_2$,\\
\emph{(2)} $q_3 = q_3' = 0$, and $z' = e^{i\frac{2\pi k}{p}}z$ for some $k = 1,\dots,p$.\\
The resulting topological space is homeomorphic to $L(p;q)$, see Fig.~\emph{\ref{pic-lens-space}}.
\end{proposition}

\begin{figure}
\centering{
\minipage{0.45\textwidth}
  \centering{\includegraphics[width=0.7\linewidth]{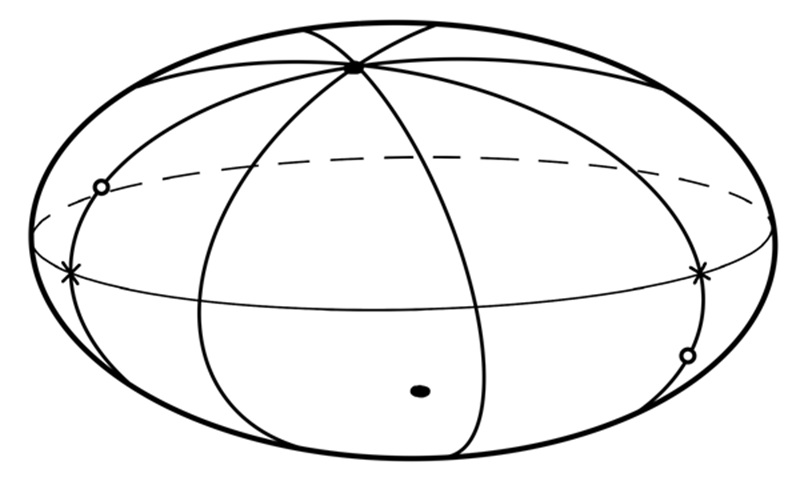}}
\endminipage
\hfil
\minipage{0.45\textwidth}
  \centering{\includegraphics[width=0.5\linewidth]{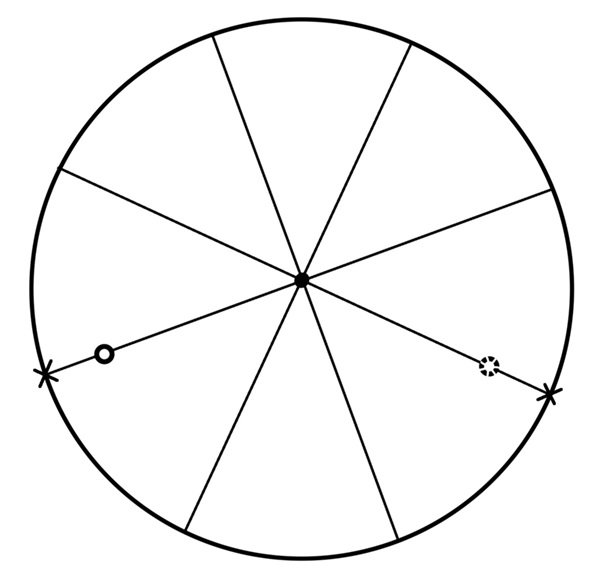}}
\endminipage}
\caption{\label{pic-lens-space}The lens space $L(p;q)$, $p > 1$ as a domain $L \subset \R^3$ with identified points on its boundary $\partial L$. Identified points have the same marks. The general view (left) and the projection to the $(q_1,q_2)$-plane (right) for $p = 8$ and $q = 3$.}
\end{figure}

\begin{remark}
\label{rem-lens-space}
It is not difficult to see that there are homeomorphisms
$$
L(p;q) \simeq L(-p;q) \simeq L(p;-q) \simeq L(p; kp + q) \qquad \text{for} \qquad k \in \Z.
$$

Let us mention some particular cases.
First, note that the Lie group
$$
\SU_2 = \left\{\left(
\begin{array}{rr}
z & w\\
-\bar{w} & \bar{z}\\
\end{array}
\right) \, \Bigm| \, z,w \in \C, \ |z|^2 + |w|^2 = 1 \right\}
$$
is homeomorphic to $S^3$.
Second, the lens spaces $L(p;1)$ are $\SU_2$-homogeneous.
Indeed, the action of the generator of the group $\Z_p$ can be represented by the left multiplication by the matrix
$$
\left(
\begin{array}{cc}
e^{i\frac{2\pi}{p}} & 0 \\
0 & e^{-i\frac{2\pi}{p}} \\
\end{array}
\right)
$$
which commutes with the right action of the group $\SU_2$.
Similarly, the lens space $L(p;-1)$ is $\SU_2$-homogeneous, since the $\Z_p$-generator's action can be represented by the right multiplication by the same matrix that commutes with the left action of the group $\SU_2$.

Obviously, $L(1;1) = S^3 = \SU_2$ and $L(2;1) = \R P^2 = \SO_3$.
The lens space $L(4;1)$ is homeomorphic to the projectivization of the tangent bundle of the two dimensional sphere.
It may be a useful model for an anthropomorphic contour recovery for spherical images avoiding cusps, compare with~\cite{duits-pse2,duits-so3}.
\end{remark}

Now we define a Riemannian metric we deal with.

Notice that the left action of the circle group
$$
\U_1 = \left\{\left(
\begin{array}{cc}
e^{i\varphi} & 0 \\
0 & e^{-i\varphi} \\
\end{array}
\right) \, \Bigm| \, \varphi \in \R \right\} \subset \SU_2
$$
on $\SU_2$ commutes with the $\Z_p$-action. This implies that the group $\U_1$ acts also on the lens space $L(p;q)$.
The factor is homeomorphic to $S^2$, since it is a complex projective line $\C P^1$ with homogeneous coordinates $[z:w]$ factorized by cyclic group generated by rotation by the angle $\frac{2\pi(1-q)}{p}$ around zero in the affine chart $w \neq 0$.

\begin{definition}
\label{def-Hopf-fibration}
The natural projection $L(p;q) \rightarrow L(p;q)/\U_1 \simeq S^2$ is called \emph{the Hopf fibration}.
Its fibers are circles.
\end{definition}

We consider the symmetric metric on $S^3$ deformed along the fibers of the Hopf fibration and then transferred to $L(p;q)$.

More precisely, we consider a $\SU_2$-left-invariant and $\U_1$-right-invariant Riemannian metric on $S^3$, in other words,
a left-invariant Riemannian metric $g$ on the Lie group $\SU_2$ that in the tangent space of the identity reads as
$$
g(u_1,u_2,u_3) = I_1 u_1^2 + I_1 u_2^2 + I_3 u_3^2, \qquad \text{where} \qquad (u_1,u_2,u_3) \in \su_2 = T_{\id}\SU_2,
$$
and $u_1,u_2,u_3$ are coordinates in the Lie algebra $\su_2$ corresponding to the basis
$$
e_1 = \frac{1}{2} \left(
\begin{array}{rr}
0 & 1\\
-1 & 0\\
\end{array}
\right), \qquad
e_2 = \frac{1}{2} \left(
\begin{array}{cc}
0 & i\\
i & 0\\
\end{array}
\right), \qquad
e_3 = \frac{1}{2} \left(
\begin{array}{rr}
i & 0\\
0 & -i\\
\end{array}
\right).
$$

\begin{remark}
\label{rem-metric}
Notice that if $I_1 = I_3$, then $g = 4 I_1 s$, where $s$ is the standard metric on the unit sphere $S^3 \subset \R^4$ (i.e., restricted from the Euclidian structure of $\R^4$).
\end{remark}

\begin{proposition}
\label{prop-Zp-isom}
The group $\Z_p$ acts on $S^3 \simeq \SU_2$ by isometries.
\end{proposition}

\begin{proof}
Define by $f : S^3 \rightarrow S^3$ the action of the generator of $\Z_p$, i.e., $f : (z,w) \mapsto (\varepsilon z, \varepsilon^q w)$,
where $\varepsilon = e^{i\frac{2\pi}{p}}$.
It is easy to check that
$$
df \circ dL_{(z,w)}(i\xi,\zeta) = dL_{f(z,w)}(i\xi,\varepsilon^{q-1}\zeta),
$$
where $L_{(z,w)}$ denotes the left-shift by an element $\left(\begin{array}{rr}z & w\\ -\bar{w} & \bar{z}\\ \end{array}\right) \in \SU_2$
and $(i\xi,\zeta)$ is a tangent vector to $\SU_2$ at the identity point, i.e., $\left(\begin{array}{rr}i\xi & \zeta\\ -\bar{\zeta} & -i\xi\\ \end{array}\right) \in \su_2$.
Since our metric is axisymmetric, i.e., $\U_1$-invariant, it follows that $(i\xi,\zeta) \mapsto (i\xi,\varepsilon^{q-1}\zeta)$ is local isometry.
\end{proof}

Proposition~\ref{prop-Zp-isom} implies that we can transfer axisymmetric Riemannian metric from $\SU_2$ to the lens space $L(p;q)$.
We call the resulting metric \emph{the Berger metric} on $L(p;q)$ by analogue with the Berger sphere.
Let as introduce a parameter that measures the oblateness of the metric:
$$
\eta = \frac{I_1}{I_3} - 1 > -1.
$$
When $\eta \rightarrow -1$, or, equivalently, $I_3 \rightarrow +\infty$, we obtain sub-Riemannian structure on $L(p;q)$.
The corresponding sub-Riemannian distribution is the projection of the left-invariant distribution on the Lie group $\SU_2$
defined by the subspace $\sspan{\{e_1, e_2\}} \subset \su_2$.
This sub-Riemannian structure was studied by U.~Boscain and F.~Rossi~\cite[Sec.~4]{boscain-rossi}.

\section{\label{sec-known-res}Some necessary definitions and known results}

Any geodesic in the lens space is a projection of a geodesic in the group $\SU_2$.
We use the Hamiltonian approach to describe geodesics in $\SU_2$, see~\cite[Sec.~2]{podobryaev-sachkov-so3} for details.
Consider the cotangent bundle $\pi : T^*\SU_2 \rightarrow \SU_2$.
Any geodesic is a projection of a trajectory of the Hamiltonian vector field $\vec{H}$ in $T^*\SU_2$ corresponding to the Hamiltonian
$H = \frac{1}{2} \left(\frac{h_1^2}{I_1} + \frac{h_2^2}{I_1} + \frac{h_3^2}{I_3}\right)$, where $h_i(\lambda) = \langle dL_{\pi(\lambda)} e_i, \,\cdot\, \rangle$, $\lambda \in T^*\SU_2$
are linear on the fibers of the cotangent bundle Hamiltonians.
Any arclength parameterized geodesic starting from the point $\id \in \SU_2$ is determined by its initial covector $h \in C \subset T^*_{\id}\SU_2 \simeq \su_2^*$,
where $C = \{h \in \su_2^* \, | \, H(h) = \frac{1}{2}\}$ is a level surface of the Hamiltonian $H$.

\begin{definition}
\label{def-expmap}
(1) \emph{The exponential map} is the map
$$
\Exp : C \times \R_+ \rightarrow L(p;q), \qquad \Exp(h,t) = \Pi \circ \pi \circ e^{t\vec{H}}h, \qquad (h,t) \in C \times \R_+,
$$
where $e^{t\vec{H}}$ is the flow of the Hamiltonian vector field $\vec{H}$ and $\Pi : \SU_2 \rightarrow L(p;q)$ is the factorization by the $\Z_p$-action.\\
(2) A pair of diffeomorphisms $s : C \times \R_+$ and $S : L(p;q) \rightarrow L(p;q)$, where $s$ keeps the time, is called \emph{a symmetry of the exponential map}
if $\Exp \circ s = S \circ \Exp$.\\
(3) A point $m \in L(p;q)$ is called \emph{a Maxwell point} if two different geodesics of the same length starting from the point $o = \Pi(\id)$ meet one another at the point $m$.\\
(4) A critical value of the exponential map is called \emph{a conjugate point}.\\
(5) If a geodesic starting from the point $o$ is optimal up to the point $c$ on it and is not optimal after this point, then the point $c$ is called \emph{a cut point} along this geodesic. The set of the cut points along all geodesics starting from the point $o$ is called \emph{the cut locus} $\Cut_o$.
\end{definition}

Usually Maxwell points appear due to symmetries of the exponential map. If $\Sym$ is some group of symmetries, then by $\tmax^{\Sym}$ we denote the first time when a Maxwell point corresponding to symmetry appears and call it \emph{the first Maxwell time for the symmetry group} $\Sym$. Analogically one can define \emph{the first conjugate time} and \emph{the cut time}. The first Maxwell time, the first conjugate time and the cut time are the functions of initial covector of a geodesic:
$$
\tmax^{\Sym} : C \rightarrow \R_+ \cup \{+\infty\}, \qquad \tconj : C \rightarrow \R_+ \cup \{+\infty\}, \qquad \tcut : C \rightarrow \R_+ \cup \{+\infty\}.
$$

We will find the set of first Maxwell points corresponding to symmetries $M^{\Sym}_o$ and prove that the exponential map is a diffeomorphism of the following domains:
$$
\Exp : \{(h,t) \in C \times \R_+ \, | \, 0 < t < \tmax^{\Sym}(h)\} \rightarrow L(p;q) \setminus \left( \closure{M^{\Sym}_o} \cup \{ o \} \right),
$$
where $\closure{M^{\Sym}_o}$ is the closure of $M^{\Sym}_o$.
This will imply that $\Cut_o = \closure{M^{\Sym}_o}$. To prove this diffeomorphism we will need the inequality $\tmax^{\Sym}(h) \leqslant \tconj(h)$ for any $h \in C$.

Also we need geodesic equations and the equations for the conjugate time obtained in~\cite{podobryaev-sachkov-so3,bates-fasso}.

Introduce the following notation. If $h = (h_1,h_2,h_3) \in \su_2^*$ are coordinates in the basis dual to the basis $e_1,e_2,e_3 \in \su_2$, then
$$
|h| = \sqrt{h_1^2 + h_2^2 + h_3^2}, \qquad \bar{h}_i = \frac{h_i}{|h|}, \ i = 1,2,3, \qquad \tau = \frac{2I_1 t}{|h|}.
$$

It is a consequence of the famous Geodesic Lemma~\cite{kowalski-vanhecke} that geodesics of our metric on $\SU_2$ are products of two one-parametric subgroups, but we need an explicit formulas in coordinates.

\begin{proposition}[\cite{podobryaev-sachkov-so3}, formula~(4)]
\label{prop-geodesics}
A geodesic starting from the point $\id \in \SU_2$ with initial covector $h \in C \subset \su_2^*$ has the following parametrization.
If $z = q_0 + iq_3$ and $w = q_1 + iq_2$, then
$$
\begin{array}{ccl}
q_0(\tau) & = & \cos{\tau}\cos{(\tau\eta\bar{h}_3)} - \bar{h}_3 \sin{\tau}\sin{(\tau\eta\bar{h}_3)}, \\
\left(
\begin{array}{c}
q_1(\tau)\\
q_2(\tau)\\
\end{array}
\right) & = & \sin{\tau} R_{-\tau\eta\bar{h}_3}
\left(
\begin{array}{c}
\bar{h}_1(\tau)\\
\bar{h}_2(\tau)\\
\end{array}
\right), \\
q_3(\tau) & = & \cos{\tau}\sin{(\tau\eta\bar{h}_3)} + \bar{h}_3 \sin{\tau}\cos{(\tau\eta\bar{h}_3)}, \\
\end{array}
$$
where $R_{\alpha}$ is the matrix of rotation by the angle $\alpha$.
\end{proposition}

\begin{proposition}[\cite{bates-fasso}]
\label{prop-conjugate-time}
The conjugate time is equal to $\tconj(h) = \frac{2I_1 \tauconj(\bar{h}_3)}{|h|}$, where $\tauconj(\bar{h}_3)$ is the minimum value of $\pi$ and the smallest positive root of the equation
\begin{equation}
\label{eq-conj}
\tan{\tau} = -\tau\eta\frac{1-\bar{h}_3^2}{1 + \eta\bar{h}_3^2}.
\end{equation}
$(1)$ If $-1 < \eta \leqslant 0$, then $\tauconj(\bar{h}_3) = \pi$ for any $\bar{h}_3 \in [-1, 1]$.\\
$(2)$ If $\eta > 0$, then $\tauconj(\bar{h}_3)$ is the smallest positive root of the equation~\emph{\eqref{eq-conj}}
and the inequality
$$
\frac{\pi}{2} < \tauconj(\bar{h}_3) \leqslant \pi
$$
is satisfied. There is the equality only for $\bar{h}_3 = \pm 1$.\\
$(3)$ The function $\tauconj$ increases on the segment $[0, 1]$ for $\eta > 0$.
\end{proposition}

\section{\label{sec-sym}Symmetries and some technical statements}

The exponential map for an axisymmetric Riemannian problem on the lens space $L(p;q)$ has the same symmetry group as in the cases of $\SU_2$ and $\SO_3$.
The group of symmetries is $\Sym \simeq \mathrm{O}_2 \times \Z_2$, we refer for details to~\cite[Sec.~4]{podobryaev-sachkov-so3}.
The generators of this group act in the pre-image and in the image of the exponential map in the following way:
\begin{enumerate*}
\item[(s1)] $s$ is a rotation around $h_3$-axis and $S$ is a rotation around $q_3$-axis in $(q_1,q_2,q_3)$-space;
\item[(s2)] $s$ is a reflection with respect to a plane containing $h_3$-axis and $S$ is a reflection with respect to a plane containing $q_3$-axis;
\item[(s3)] $s$ is the reflection with respect to the $(h_1,h_2)$-plane and $S$ is the reflection with respect to the $(q_1,q_2)$-plane.
\end{enumerate*}
It is already known that the first Maxwell time $\tmax^{\Sym}$ for $p = 2$ depends on the Maxwell time corresponding to the rotations~(s1) and compositions of the rotations~(s1) with the reflection~(s3), see~\cite[Prop.~4]{podobryaev-sachkov-so3}. It is easy to see from the geodesic equations (Proposition~\ref{prop-geodesics}) that the Maxwell time for rotations corresponds to $\tau = \pi$, while the compositions of the rotations by angles $\frac{2\pi kq}{p}$ with the reflection~(s3) give Maxwell points that are determined by topology of the lens space $L(p;q)$. These Maxwell points form the surface $\partial L$ with identified points as described in Proposition~\ref{prop-lens-space}.
The corresponding $\tau$ for this Maxwell time is the first positive root with respect to the variable $\tau$ of the equation of the surface $\partial L$ where the parametric geodesics equations are substituted.
Thus, the comparison of this first positive root with $\pi$ plays a key role in the analysis of Maxwell points.

Remember that the surface $\partial L \subset S^3$ for $p > 1$ has the equation
$$
q_1^2 + q_2^2 + \frac{q_3^2}{\sin^2{\frac{\pi}{p}}} - 1 = - q_0^2 - q_3^2 + \frac{q_3^2}{\sin^2{\frac{\pi}{p}}} =
\frac{1}{\sin^2{\frac{\pi}{p}}} \left( q_3^2 \cos^2{\frac{\pi}{p}} - q_0^2 \sin^2{\frac{\pi}{p}} \right) =
$$
$$
= \frac{1}{\sin^2{\frac{\pi}{p}}} \left(q_3 \cos{\frac{\pi}{p}} - q_0 \sin{\frac{\pi}{p}}\right) \left(q_3 \cos{\frac{\pi}{p}} + q_0 \sin{\frac{\pi}{p}}\right) = 0.
$$
Let us introduce the function $\ell(q) = \ell_-(q) \ell_+(q)$, where
\begin{equation}
\label{eq-ell}
\begin{array}{cclcl}
\ell_-(q) & = & q_3 \cos{\frac{\pi}{p}} - q_0 \sin{\frac{\pi}{p}} & = &
\cos{\tau}\sin{\left(\tau\eta\bar{h}_3 - \frac{\pi}{p} \right)} + \bar{h}_3 \sin{\tau}\cos{\left(\tau\eta\bar{h}_3 - \frac{\pi}{p} \right)}, \\
\ell_+(q) & = & q_3 \cos{\frac{\pi}{p}} + q_0 \sin{\frac{\pi}{p}} & = &
\cos{\tau}\sin{\left(\tau\eta\bar{h}_3 + \frac{\pi}{p} \right)} + \bar{h}_3 \sin{\tau}\cos{\left(\tau\eta\bar{h}_3 + \frac{\pi}{p} \right)}.\\
\end{array}
\end{equation}
Denote by $\taul$ the first positive root of $\ell(q_0(\bar{h}_3,\tau),q_1(\bar{h}_3,\tau),q_2(\bar{h}_3,\tau),q_3(\bar{h}_3,\tau))$ with respect to variable $\tau$ and fixed $\bar{h}_3$,
Obviously, $\taul = \min{(\taul^-, \taul^+)}$, where $\taul^-$  and $\taul^+$ are the first positive roots of the equations (respectively):
$$
\ell_-(q_0(\bar{h}_3,\tau),q_1(\bar{h}_3,\tau),q_2(\bar{h}_3,\tau),q_3(\bar{h}_3,\tau)) = 0, \qquad
\ell_+(q_0(\bar{h}_3,\tau),q_1(\bar{h}_3,\tau),q_2(\bar{h}_3,\tau),q_3(\bar{h}_3,\tau)) = 0.
$$
We regard $\taul^-$ and $\taul^+$ as functions of $\bar{h}_3$, so $\taul^-, \taul^+ : [-1,1] \setminus \{0\} \rightarrow \R_+$.

\begin{remark}
\label{rem-p1-domain}
Note that for $p > 1$ the functions $\taul^+$ and $\taul^-$ are defined on the whole interval $[-1,1]$.
While in the case $p = 1$ these functions are undefined for $\bar{h}_3 = 0$.
\end{remark}

\begin{remark}
\label{rem-p1}
Note that if $p = 1$, then $\ell(q) = q_3^2$ and $\taul(\bar{h}_3) = \taul^-(\bar{h}_3) = \taul^+(\bar{h}_3)$ is the first positive root of the equation $q_3(\bar{h}_3, \tau) = 0$. This Maxwell time corresponding to the reflection with respect to the plane $h_3 = 0$, i.e., the symmetry~(s3), plays the key role in the Maxwell time analysis in the case $p = 1$, see~\cite[Prop.~8--10]{podobryaev-sachkov-so3}. So, the notation $\taul$ is universal and doesn't depend on $p$, while the surface $\ell = 0$ is the surface of fixed point for different symmetries that depend on $p$.
\end{remark}

We need several technical statements.

\begin{lemma}
\label{lem-taulminus-taulplus}
The function $\taul$ is even and the following expression for $\taul$ is satisfied\emph{:}
$$
\taul(\bar{h}_3) = \left\{
\begin{array}{lcl}
\taul^-(\bar{h}_3), & \text{if} & \bar{h}_3 > 0,\\
\taul^+(\bar{h}_3), & \text{if} & \bar{h}_3 < 0,\\
\end{array}
\right.
\qquad \text{and} \qquad \taul^+(-\bar{h}_3) = \taul^-(\bar{h}_3).
$$
\end{lemma}

\begin{proof}
Note that $\ell_+(-\bar{h}_3) = -\ell_-(\bar{h}_3)$. This implies that $\taul^+(-\bar{h}_3) = \taul^-(\bar{h}_3)$, see Fig.~\ref{pic-taulminus-taulplus}.
If $p = 1$ or $p = 2$, then $\taul^-(\bar{h}_3) = \taul^+(\bar{h}_3)$.
Assume that $p > 2$. We have $\taul^-(0) = \taul^+(0) = \frac{\pi}{2}$ and $\taul^-(1) = \frac{\pi}{p(1+\eta)} < \frac{(p-1)\pi}{p(1+\eta)} = \taul^+(1)$.
If $\taul^-(\bar{h}_3) = \taul^+(\bar{h}_3)$ for some $\bar{h}_3$, then $q_0(\bar{h}_3) = q_3(\bar{h}_3) = 0$, this implies that $\bar{h}_3 = 0$ (see~\cite[Lem.~1]{podobryaev-sachkov-so3} for details).
So, $\taul^-(\bar{h}_3) \leqslant \taul^+(\bar{h}_3)$ for $\bar{h}_3 > 0$ and $\taul^-(\bar{h}_3) \geqslant \taul^+(\bar{h}_3)$ for $\bar{h}_3 < 0$.
\end{proof}

\begin{proposition}
\label{prop-cont}
The functions $\taul^-$ and $\taul^+$ are continuous on their domains of definition.
\end{proposition}

\begin{proof}
We may prove this statement only for $\taul^-$ thanks to Lemma~\ref{lem-taulminus-taulplus}.

It is sufficient to prove that for any $\bar{h}_3 \in [-1,1]$ there is no $\tau$ such that
$\ell_-|_{\bar{h}_3,\tau} = 0$ and $\frac{\partial \ell_-}{\partial \tau}|_{\bar{h}_3,\tau} = 0$.

Assume by contradiction that there is such a $\tau$.
From~\eqref{eq-ell} we obtain
\begin{equation}
\label{eq-zero}
\begin{array}{ccl}
\ell_-|_{\bar{h}_3,\tau} & =  & \cos{\tau}\sin{\left(\tau\eta\bar{h}_3 - \frac{\pi}{p}\right)} + \bar{h}_3\sin{\tau}\cos{\left(\tau\eta\bar{h}_3 - \frac{\pi}{p}\right)} = 0,\\
\frac{\partial \ell_-}{\partial \tau}|_{\bar{h}_3,\tau} & = &
-(1+\eta\bar{h}_3^2) \sin{\tau}\sin{\left(\tau\eta\bar{h}_3 - \frac{\pi}{p}\right)} + \bar{h}_3(1+\eta)\cos{\tau}\cos{\left(\tau\eta\bar{h}_3 - \frac{\pi}{p}\right)} = 0.\\
\end{array}
\end{equation}
If $p = 1$, then $\bar{h}_3 \neq 0$, see Remark~\ref{rem-p1-domain}. If $p > 1$ consider the case $\bar{h}_3 = 0$. It follows that $\cos{\tau} = \sin{\tau} = 0$ and we get a contradiction.
So, we can assume that $\bar{h}_3 \neq 0$.

Assume that $\cos{\tau}\cos{\left(\tau\eta\bar{h}_3 - \frac{\pi}{p}\right)} = 0$.
If $\cos{\tau} = 0$, then, since $\bar{h}_3 \neq 0$ and $\sin{\tau} \neq 0$, from the first equation we get $\cos{\left(\tau\eta\bar{h}_3 - \frac{\pi}{p}\right)} = 0$.
Whence, from the second equation, since $\sin{\left(\tau\eta\bar{h}_3 - \frac{\pi}{p}\right)} \neq 0$ and $1 + \eta\bar{h}_3^2 > 0$ we get $\sin{\tau} = 0$.
We get a contradiction. The same arguments imply that $\cos{\left(\tau\eta\bar{h}_3 - \frac{\pi}{p}\right)} \neq 0$.
So, we may divide both equations by $\cos{\tau}\cos{\left(\tau\eta\bar{h}_3 - \frac{\pi}{p}\right)}$. We obtain
$$
\begin{array}{l}
\tan{\left(\tau\eta\bar{h}_3 - \frac{\pi}{p}\right)} + \bar{h}_3\tan{\tau} = 0,\\
-(1+\eta\bar{h}_3^2) \tan{\tau}\tan{\left(\tau\eta\bar{h}_3 - \frac{\pi}{p}\right)} + \bar{h}_3(1+\eta) = 0.\\
\end{array}
$$
It follows that
$$
(1+\eta\bar{h}_3^2) \bar{h}_3 \tan^2{\tau} + \bar{h}_3(1+\eta) = 0, \qquad \bar{h}_3 \neq 0 \qquad \Rightarrow \qquad (1+\eta\bar{h}_3^2) \tan^2{\tau} = -(1+\eta),
$$
but $1+\eta > 0$ and $1 + \eta\bar{h}_3^2 > 0$. We get a contradiction.
\end{proof}

\begin{lemma}
\label{lem-pgr1-etagr0}
Assume that $p > 1$ and $\eta > 0$. Then $\taul(\bar{h}_3) \leqslant \frac{\pi}{2}$ where there is the equality only for $\bar{h}_3 = 0$.
\end{lemma}

\begin{proof}
It is easy to see that $\taul^-(0) = \frac{\pi}{2}$.
Thanks to Lemma~\ref{lem-taulminus-taulplus} it is sufficient to show that $\taul^-(\bar{h}_3) < \frac{\pi}{2}$ for $\bar{h}_3 > 0$.

First, note that the first positive root of the equation $\ell_-|_{\bar{h}_3 = 1} = \sin{\left( \tau(1+\eta) - \frac{\pi}{p}\right)} = 0$ equals
$\taul^-(1) = \frac{\pi}{p(1+\eta)} < \frac{\pi}{2}$.
Assume that there exists $\bar{h}_3 \in (0,1)$ such that $\taul^-(\bar{h}_3) > \frac{\pi}{2}$.
Since the function $\taul^-$ is continuous by Proposition~\ref{prop-cont}, then there exists $\hat{h}_3 \in (\bar{h}_3, 1)$ such that
$\taul^-(\hat{h}_3) = \frac{\pi}{2}$.
This mean that
\begin{equation}
\label{eq-discrete-series}
\ell_-|_{\bar{h}_3 = \hat{h}_3, \tau = \frac{\pi}{2}} = \hat{h}_3 \cos{\left( \frac{\pi\eta\hat{h}_3}{2} - \frac{\pi}{p} \right)} = 0 \qquad \Rightarrow \qquad
\hat{h}_3 = \frac{(2k+1)p + 2}{p\eta}, \quad k \in \Z.
\end{equation}
Second, consider the arc of the graph of the function $\taul^-$ between the point $(0,\frac{\pi}{2})$
and the point $x_k = \left(\frac{(2k+1)p + 2}{p\eta}, \frac{\pi}{2}\right)$ where $k$ is the minimal integer such that $0 < \frac{(2k+1)p + 2}{p\eta} < 1$ and $x_k$ belongs to the graph of the function $\taul^-$. This implies that $k > -\frac{p+2}{2p} \geqslant -1$.
We claim that $k = 0$.
Indeed, otherwise $\taul^-\left(\frac{(2m+1)p + 2}{p\eta}\right)$ for $0 \leqslant m < k$ is not the first positive root of the equation
$\ell_-|_{\bar{h}_3 = \frac{(2m+1)p + 2}{p\eta}, \tau = \frac{\pi}{2}} = 0$,
since for the derivatives of the function $\ell_-$ at these points from formula~\eqref{eq-zero} we have
$$
\frac{d\ell_-}{d\tau}\left(0,\frac{\pi}{2}\right) = \sin{\frac{\pi}{p}} > 0, \qquad \frac{d\ell_-}{d\tau}(x_0) = -(1 + \eta\hat{h}_3^2) < 0.
$$
Since the function $\ell_-$ is smooth there exists a point
$(\bar{h}_3, \taul^-(\bar{h}_3))$ on this arc such that $\frac{\partial \ell_-}{\partial \tau} = 0$ at this point. Moreover, $\ell_- = 0$ at this point as well.
But this is not possible as follows from the proof of Proposition~\ref{prop-cont}, see~\eqref{eq-zero}.
We get a contradiction.
\end{proof}

\begin{lemma}
\label{lem-dtau}
The following equality is satisfied\emph{:}
$$
\frac{d \taul^-}{d \bar{h_3}} = -\frac{\taul^- \eta \cos{\taul^-}\cos{\left( \taul^-\eta\bar{h}_3 - \frac{\pi}{p} \right)} +
\sin{\taul^-}\cos{\left( \taul^-\eta\bar{h}_3 - \frac{\pi}{p} \right)} -
\taul^-\eta\bar{h}_3 \sin{\taul^-}\sin{\left( \taul^-\eta\bar{h}_3 - \frac{\pi}{p} \right)}}{-(1+\eta\bar{h}_3^2)\sin{\taul^-}\sin{\left( \taul^-\eta\bar{h}_3 - \frac{\pi}{p} \right)} + \bar{h}_3(1+\eta)\cos{\taul^-}\cos{\left( \taul^-\eta\bar{h}_3 - \frac{\pi}{p} \right)}}.
$$
\end{lemma}

\begin{proof}
It is a direct computation of $\frac{d \taul^-}{d \bar{h_3}} = - \frac{\partial \ell_-}{\partial \bar{h}_3} \Bigm\slash \frac{\partial \ell_-}{\partial \taul^-}$
using the second formula of~\eqref{eq-zero} and
$$
\frac{d\ell_-}{d\bar{h}_3} = \tau\eta\cos{\tau}\cos{\left(\tau\eta\bar{h}_3 - \frac{\pi}{p}\right)} + \sin{\tau}\cos{\left(\tau\eta\bar{h}_3 - \frac{\pi}{p}\right)} -
\tau\eta\bar{h}_3\sin{\tau}\sin{\left(\tau\eta\bar{h}_3 - \frac{\pi}{p}\right)}.
$$
\end{proof}

\begin{lemma}
\label{lem-decrease-special}
Assume that $\eta > 0$. If $\bar{h}_3 > 0$ is such that $\cos{\taul^-(\bar{h}_3)} = 0$ or/and $\cos{\left(\taul^-(\bar{h}_3)\eta\bar{h}_3 - \frac{\pi}{p}\right)} = 0$,
then $\frac{d \taul^-}{d \bar{h_3}}(\bar{h}_3) < 0$.
\end{lemma}

\begin{proof}
Since $\taul^-(\bar{h}_3)$ is a root of the equation $\ell_- = 0$, then
$\cos{\taul^-(\bar{h}_3)} = 0$ iff $\cos{\left(\taul^-(\bar{h}_3)\eta\bar{h}_3 - \frac{\pi}{p}\right)} = 0$.
It follows from Lemma~\ref{lem-dtau} that $\frac{d \taul^-}{d \bar{h_3}}(\bar{h}_3) = -\frac{\taul^-(\bar{h}_3)\eta\bar{h}_3}{1+\eta\bar{h}_3^2} < 0$ for $\eta > 0$.
\end{proof}

\begin{lemma}
\label{lem-dtau-tg}
If $\bar{h}_3$ is such that $\cos{\taul^-(\bar{h}_3)}\cos{\left(\taul^-(\bar{h}_3)\eta\bar{h}_3 - \frac{\pi}{p}\right)} \neq 0$, then
$$
\frac{d \taul^-}{d \bar{h_3}} = -\frac{\taul^- \eta + \tan{\taul^-} +
\taul^-\eta\bar{h}_3^2 \tan^2{\taul^-}}{\bar{h}_3(1+\eta\bar{h}_3^2)\tan^2{\taul^-} + \bar{h}_3(1+\eta)}.
$$
\end{lemma}

\begin{proof}
Dividing by $\cos{\taul^-}\cos{\left( \taul^-\eta\bar{h}_3 - \frac{\pi}{p} \right)} \neq 0$ the nominator and the denominator of the expression from Lemma~\ref{lem-dtau}
we obtain
$$
\frac{d \taul^-}{d \bar{h_3}} = -\frac{\taul^- \eta + \tan{\taul^-} -
\taul^-\eta\bar{h}_3 \tan{\taul^-}\tan{\left( \taul^-\eta\bar{h}_3 - \frac{\pi}{p} \right)}}{-(1+\eta\bar{h}_3^2)\tan{\taul^-}\tan{\left( \taul^-\eta\bar{h}_3 - \frac{\pi}{p} \right)} + \bar{h}_3(1+\eta)}.
$$
But $\ell_-(\taul^-) = 0$ implies that $\tan{\left( \taul^-\eta\bar{h}_3 - \frac{\pi}{p} \right)} = -\bar{h}_3 \tan{\taul^-}$.
Hence, we get requiring expression for $\frac{d \taul^-}{d \bar{h_3}}$.
\end{proof}

\begin{lemma}
\label{lem-decrease-1}
If $\eta > 0$ and $p > 1$, then the function $\taul^-$ decreases on the interval $(0,1]$.
\end{lemma}

\begin{proof}
If $\cos{\taul^-}\cos{\left( \taul^-\eta\bar{h}_3 - \frac{\pi}{p} \right)} = 0$, then by Lemma~\ref{lem-decrease-special} we have $\frac{d \taul^-}{d \bar{h}_3} < 0$.

So, we may assume that $\cos{\taul^-}\cos{\left( \taul^-\eta\bar{h}_3 - \frac{\pi}{p} \right)} \neq 0$ and use the expression from Lemma~\ref{lem-dtau-tg}.
First, note that the sign of the denominator is positive.
Second, if $p > 1$ then $\taul^- < \frac{\pi}{2}$ by Lemma~\ref{lem-pgr1-etagr0}. Hence, $\tan{\taul^-} > 0$ and $\taul^- \eta + \tan{\taul^-} +
\taul^-\eta\bar{h}_3^2 \tan^2{\taul^-} > 0$.
It follows that $\frac{d \taul^-}{d \bar{h}_3} < 0$ and the function $\taul^-$ decreases on the interval $(0,1]$.
\end{proof}

\begin{lemma}
\label{lem-p1-limit}
Assume that $p = 1$ and $\eta > 0$. Then $\taul^-(\bar{h}_3) \rightarrow \tauconj(0)-$ when $\bar{h}_3 \rightarrow 0+$.
\end{lemma}

\begin{proof}
We may assume that $\bar{h}_3$ is such that $\cos{\taul^-(\bar{h}_3)}\cos{(\taul^-(\bar{h}_3)\eta\bar{h}_3)} \neq 0$,
since $\bar{h}_3 \neq \frac{(2k+1)p+2}{p\eta}$, see~\eqref{eq-discrete-series}.
It is easy to see that $\taul^-(\bar{h}_3)$ is the first positive root of the equation $\frac{1}{\bar{h}_3}\tan{\tau\eta\bar{h}_3} = -\tan{\tau}$.
This equation tends to the equation $\tau\eta = -\tan{\tau}$ when $\bar{h}_3 \rightarrow 0+$.
While the last equation is the equation for $\tauconj(0)$, see Proposition~\ref{prop-conjugate-time}\,(2).
Moreover, $\taul^-(\bar{h}_3) < \tauconj(0)$ for small $\bar{h}_3 > 0$ since $\tan{\tau\eta\bar{h}_3} > \tau\eta$.
\end{proof}

\begin{lemma}
\label{lem-p1-max-conj}
If $p = 1$ and $\eta > 0$, then $\taul^-(\bar{h}_3) < \tauconj(\bar{h_3})$ for $\bar{h}_3 \in (0,1]$.
\end{lemma}

\begin{proof}
Note that $\taul^-(1) = \frac{\pi}{1+\eta} < \pi = \tauconj(1)$ (see Proposition~\ref{prop-conjugate-time}).
Also the requiring inequality is satisfied for $\bar{h}_3$ small enough by Lemma~\ref{lem-p1-limit}.
Assume by contradiction that there exist $\bar{h}_3 \in (0,1)$ such that $\taul^-(\bar{h}_3) > \tauconj(\bar{h}_3)$.
Since the functions $\taul^-$ and $\tauconj$ are continuous, there exists at least two points $\hat{h}_3$ such that $\taul^-(\hat{h}_3) = \tauconj(\hat{h}_3)$.

We claim that at these points $\frac{d\taul^-}{d\bar{h}_3}(\hat{h}_3) < 0$ and this will give a contradiction.
Indeed, since $\taul^-(\hat{h}_3) = \tauconj(\hat{h}_3)$, by Proposition~\ref{prop-conjugate-time}\,(2) we obtain $\cos{\taul^-(\hat{h}_3)} \neq 0$ and
we can apply Lemma~\ref{lem-dtau-tg}.
From the formula for the conjugate time~\eqref{eq-conj} it follows that
$\taul^-(\hat{h}_3)\eta + \tan{\taul^-(\hat{h}_3)} = \taul^-(\hat{h}_3)\eta\frac{(1+\eta)\hat{h}_3^2}{1 + \eta\hat{h}_3^2} > 0$.
So, the nominator of the formula from Lemma~\ref{lem-dtau-tg} is positive, the denominator is positive as well.
Hence, $\frac{d\taul^-}{d\bar{h}_3}(\hat{h}_3) < 0$.

It remains to consider the case when there exist a unique point $\hat{h}_3$ such that $\taul^-(\hat{h}_3) = \tauconj(\hat{h}_3)$.
It this case $\frac{d\taul^-}{d\bar{h}_3}(\hat{h}_3) = \frac{d\tauconj}{d\bar{h}_3}(\hat{h}_3) > 0$ since the function $\tauconj$ increases by Proposition~\ref{prop-conjugate-time}\,(3). But we proved that $\frac{d\taul^-}{d\bar{h}_3}(\hat{h}_3) < 0$. This contradiction completes the proof.
\end{proof}

\section{\label{sec-maxwell}The Maxwell time corresponding to symmetries}

Now we are ready to compare the Maxwell time corresponding to different kinds of symmetries.

\begin{figure}
\centering{
\minipage{0.45\textwidth}
  \center{\includegraphics[width=0.6\linewidth]{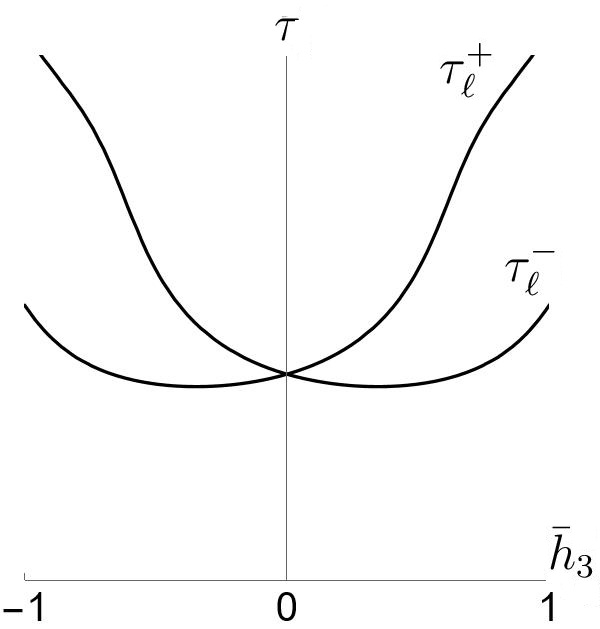}}
  \caption{\label{pic-taulminus-taulplus} The functions $\taul^-, \taul^+ : [0,1] \rightarrow \R_+$.}
\endminipage
\hfil
\minipage{0.45\textwidth}
  \center{\includegraphics[width=0.6\linewidth]{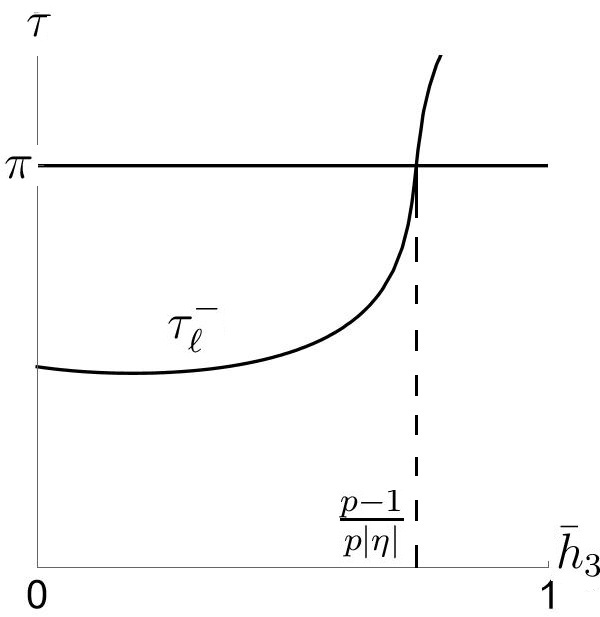}}
  \caption{\label{pic-pi-taulminus} The function $\taul^-$ and $\pi$ in the case $\eta < -\frac{p-1}{p}$ and $p > 1$.}
\endminipage}
\end{figure}

\begin{proposition}
\label{prop-maxwell-time-comparsion-gr1}
Assume that $p > 1$.\\
\emph{(1)} If $\eta \geqslant -\frac{p-1}{p}$, then $\taul(\bar{h}_3) \leqslant \pi$ for any $\bar{h}_3 \in [-1,1]$.\\
\emph{(2)} If $\eta < -\frac{p-1}{p}$, then $\taul(\bar{h}_3) \geqslant \pi$ for $\frac{p-1}{p|\eta|} \leqslant |\bar{h}_3| \leqslant 1$
and $\taul(\bar{h}_3) < \pi$ for $|\bar{h}_3| < \frac{p-1}{p|\eta|}$, see Fig.~\emph{\ref{pic-pi-taulminus}}.
\end{proposition}

\begin{proof}
Note that $\ell_-|_{\tau = 0} = -\sin{\frac{\pi}{p}} < 0$ and $\ell_+|_{\tau = 0} = \sin{\frac{\pi}{p}} > 0$ when $p > 1$.
So, to prove~(1) and the second part of~(2) it is sufficient to find a function $\theta_{\eta} : [-1,1] \rightarrow \R$ such that $\theta_{\eta}(\bar{h}_3) \leqslant \pi$
and $\ell_-|_{\tau = \theta_{\eta}(\bar{h}_3)} \geqslant 0$ or $\ell_+|_{\tau = \theta_{\eta}(\bar{h}_3)} \leqslant 0$.
Indeed, since the functions $\ell_-$ and $\ell_+$ of variable $\tau$ are continuous, the function $\ell$ has a root on the interval $(0,\theta_{\eta}(\bar{h}_3)] \subset (0,\pi]$.

Let us choose
$$
\theta_{\eta}(\bar{h}_3) = \left\{
\begin{array}{ccl}
\pi, & \text{if} & |\eta\bar{h}_3| \leqslant \frac{p-1}{p},\\
\frac{(p-1)\pi}{p|\eta\bar{h}_3|}, & \text{if}  & |\eta\bar{h}_3| > \frac{p-1}{p}.
\end{array}
\right.
$$
If $|\eta\bar{h}_3| \leqslant \frac{p-1}{p}$, then
$$
\ell_-|_{\tau = \theta_{\eta}(\bar{h}_3)} = \ell_-|_{\tau = \pi} = -\sin{\left(\pi\eta\Bar{h}_3 - \frac{\pi}{p}\right)}, \qquad
\ell_+|_{\tau = \theta_{\eta}(\bar{h}_3)} = \ell_+|_{\tau = \pi} = -\sin{\left(\pi\eta\Bar{h}_3 + \frac{\pi}{p}\right)}.
$$
If $-\frac{p-1}{p} \leqslant \eta\bar{h}_3 \leqslant \frac{1}{p}$, then $-\pi \leqslant \pi\eta\bar{h}_3 - \frac{\pi}{p} \leqslant 0$ and
$\ell_-|_{\tau = \theta_{\eta}(\bar{h}_3)} = -\sin{\left(\pi\eta\Bar{h}_3 - \frac{\pi}{p}\right)} \geqslant 0$.
If $-\frac{1}{p} \leqslant \eta\bar{h}_3 \leqslant \frac{p-1}{p}$, then $0 \leqslant \pi\eta\bar{h}_3 + \frac{\pi}{p} \leqslant \pi$ and
$\ell_+|_{\tau = \theta_{\eta}(\bar{h}_3)} = -\sin{\left(\pi\eta\Bar{h}_3 + \frac{\pi}{p}\right)} \leqslant 0$.

Consider now the case $|\eta\bar{h}_3| > \frac{p-1}{p}$ and $\eta > 0$. By formula~\eqref{eq-ell} we have
$$
\begin{array}{rcl}
\bar{h}_3 < 0 & \Rightarrow & \ell_-|_{\tau = \theta_{\eta}(\bar{h}_3)} = -\bar{h}_3 \sin{\frac{(p-1)\pi}{p|\eta\bar{h}_3|}} > 0, \\
\Bar{h}_3 > 0 & \Rightarrow & \ell_+|_{\tau = \theta_{\eta}(\bar{h}_3)} = -\bar{h}_3 \sin{\frac{(p-1)\pi}{p|\eta\bar{h}_3|}} < 0, \\
\end{array}
$$
This implies~(1) and the second part of~(2).

It remains to consider the case $|\eta\bar{h}_3| > \frac{p-1}{p}$ and $-1 < \eta < -\frac{p-1}{p}$.
Let us prove that $\taul(\bar{h}_3) > \pi$ for $|\bar{h}_3| > \frac{p-1}{p|\eta|}$.
It is not difficult to see that
$$
\begin{array}{lclll}
\ell_-|_{\bar{h}_3 = \pm 1} = \sin{\left(\pm\tau(1+\eta) - \frac{\pi}{p}\right)} & \Rightarrow &
\taul^-(1) = \frac{\pi}{p(1+\eta)} > \pi, & \taul^-(-1) = \frac{(p-1)\pi}{p(1+\eta)} > \pi,\\
\ell_+|_{\bar{h}_3 = \pm 1} = \sin{\left(\pm\tau(1+\eta) + \frac{\pi}{p}\right)} & \Rightarrow &
\taul^+(1) = \frac{(p-1)\pi}{p(1+\eta)} > \pi, & \taul^+(-1) = \frac{\pi}{p(1+\eta)} > \pi.\\
\end{array}
$$
This means that $\taul(\pm 1) > \pi$.

Assume by contradiction that there exists $\hat{h}_3$ such that $\frac{p-1}{p|\eta|} < |\hat{h}_3| < 1$ and
$$
\ell_{\pm}|_{\bar{h}_3 = \hat{h}_3, \tau = \pi} = - \sin{\left( \pi\eta\hat{h}_3 \pm \frac{\pi}{p} \right)} = 0.
$$
Solving this equation with respect to $\hat{h}_3$, we get $|\hat{h}_3| = \frac{|kp \pm 1|}{p|\eta|}$ where $k \in \Z$.
But any point of this series lies outside the interval $\left( \frac{p-1}{p|\eta|}, 1 \right)$. So, we get a contradiction.
\end{proof}

\begin{proposition}
\label{prop-maxwell-time-comparsion-1}
Assume that $p = 1$.\\
\emph{(1)} If $\eta < 0$, then $\taul(\bar{h}_3) \geqslant \pi$ for any $\bar{h}_3 \in [-1,1] \setminus \{0\}$.\\
\emph{(2)} If $\eta \geqslant 0$, then $\taul(\bar{h}_3) \leqslant \pi$ for any $\bar{h}_3 \in [-1,1] \setminus \{0\}$.
\end{proposition}

\begin{proof}
It is sufficient to prove these statements for the function $\taul^-$ and $\bar{h}_3 > 0$ due to Lemma~\ref{lem-taulminus-taulplus}.

(1) Note that $\taul^-(1) = \frac{\pi}{1+\eta} > \pi$ since $0 < 1 + \eta < 1$ for $-1 < \eta < 0$.
Assume by contradiction that there exists $\bar{h}_3 \in (0,1)$ such that $\taul^-(\bar{h}_3) < \pi$.
Since the function $\taul^-$ is continuous (see Proposition~\ref{prop-cont}), there exists $\hat{h}_3 \in (0,1)$ such that $\taul^-(\hat{h}_3) = \pi$.
Whence, $\ell_-|_{\bar{h}_3 = \hat{h}_3, \tau = \pi} = \sin{\pi\eta\hat{h}_3}$. It follows that $\pi\eta\hat{h}_3 = \pi k$ where $k \in \Z$.
We obtain $\hat{h}_3 = \frac{k}{\eta}$, but $\hat{h}_3 \notin (0,1)$ since $-1 < \eta < 0$. We get a contradiction.

(2) It is easy to see that if $\eta = 0$, then $\taul^- = \pi$. Consider $\eta > 0$.
It follows from Lemma~\ref{lem-p1-max-conj} that $\taul^-(\bar{h}_3) < \tauconj(\bar{h}_3)$ for $\eta > 0$ and $\bar{h}_3 > 0$.
Moreover, due to Proposition~\ref{prop-conjugate-time}\,(2) we get $\tauconj(\bar{h}_3) \leqslant \pi$.
Hence, $\taul^-(\bar{h}_3) < \pi$.
\end{proof}

\begin{proposition}
\label{prop-maxwell-time}
The first Maxwell time corresponding to symmetries is equal to
\emph{(}taking in account that $\bar{h}_3 \neq 0$ for $p = 1$\emph{)}
\begin{equation*}
\begin{array}{c}
\tmax^{\Sym}(\bar{h}_3) = \frac{2I_1\taul(\bar{h}_3)}{|h|}, \qquad \text{for} \qquad \eta \geqslant -\frac{p-1}{p}, \\ \\
\tmax^{\Sym}(\bar{h}_3) = \left\{
\begin{array}{lcl}
\frac{2I_1\pi}{|h|}, & \text{if} & |\bar{h}_3| \geqslant \frac{p-1}{p|\eta|},\\
\frac{2I_1\taul(\bar{h}_3)}{|h|}, & \text{if} & |\bar{h}_3| < \frac{p-1}{p|\eta|},\\
\end{array}
\right.
\qquad \text{for} \qquad -1 < \eta < -\frac{p-1}{p}.\\
\end{array}
\end{equation*}
\end{proposition}

\begin{proof}
Immediately follows from Propositions~\ref{prop-maxwell-time-comparsion-gr1}, \ref{prop-maxwell-time-comparsion-1}.
\end{proof}

\begin{proposition}
\label{prop-maxwell-conj}
The first Maxwell time corresponding to symmetries is less or equal to the conjugate time.
\end{proposition}

\begin{proof}
From Proposition~\ref{prop-maxwell-time}, Lemma~\ref{lem-taulminus-taulplus} and Proposition~\ref{prop-conjugate-time}\,(1) it follows that it is sufficient to prove that for $\eta > 0$ the inequality $\taul^-(\bar{h}_3) \leqslant \tauconj(\bar{h}_3)$ is satisfied for $\bar{h_3} > 0$.

For $p=1$ this follows from Lemma~\ref{lem-p1-max-conj}.
Assume now that $p > 1$.
Thanks to Proposition~\ref{prop-conjugate-time}\,(2) it is sufficient to show that $\taul^-(\bar{h}_3) \leqslant \frac{\pi}{2}$ for $\bar{h}_3 \geqslant 0$ and $\eta > 0$.
This follows from Lemma~\ref{lem-pgr1-etagr0}.
\end{proof}

\section{\label{sec-cutlocus}The cut locus and the cut time}

\begin{theorem}
\label{th-cutlocus}
\emph{(1)} The cut time with respect to the initial point $o = \Pi(\id)$ for the Berger lens space $L(p;q)$ is equal to the first Maxwell time, i.e.,
$$
\begin{array}{c}
\tcut(\bar{h}_3) = \frac{2I_1\taul(\bar{h}_3)}{|h|}, \qquad \text{for} \qquad \eta \geqslant -\frac{p-1}{p}, \\ \\
\tcut(\bar{h}_3) = \left\{
\begin{array}{lcl}
\frac{2I_1\pi}{|h|}, & \text{if} & |\bar{h}_3| \geqslant \frac{p-1}{p|\eta|},\\
\frac{2I_1\taul(\bar{h}_3)}{|h|}, & \text{if} & |\bar{h}_3| < \frac{p-1}{p|\eta|},\\
\end{array}
\right.
\qquad \text{for} \qquad -1 < \eta < -\frac{p-1}{p}.\\
\end{array}
$$
Moreover,
$$
\taul(\bar{h}_3) = \left\{
\begin{array}{lcl}
\taul^-(\bar{h}_3), & \text{if} & \bar{h}_3 \geqslant 0,\\
\taul^+(\bar{h}_3), & \text{if} & \bar{h}_3 < 0.\\
\end{array}
\right.
$$
\emph{(2)} The cut locus $\Cut_o$ with respect to the initial point $o = \Pi(\id)$ of the Berger lens space $L(p;q)$, $p > 1$ is equal to\\
\begin{enumerate}
\item[\emph{(a)}] $\partial L / \sim$ for $\eta \geqslant -\frac{p-1}{p}$
\emph{(}see Proposition~\emph{\ref{prop-lens-space}} for the definition of the equivalence relation $\sim$\emph{)}\emph{;}
\item[\emph{(b)}] the wedge sum of $\partial L / \sim$ and the interval
$\left[-\sin{\frac{\pi}{p}},-\sin{(-\pi\eta)}\right] \cup \left[\sin{(-\pi\eta)}, \sin{\frac{\pi}{p}}\right]$ on the $q_3$-axis for $-1 < \eta < -\frac{p-1}{p}$,
see Fig.~\emph{\ref{pic-lens-space-R-cut-locus}}.
\end{enumerate}
\end{theorem}

\begin{figure}[h]
\centering{
\minipage{0.45\textwidth}
  \center{\includegraphics[width=0.7\linewidth]{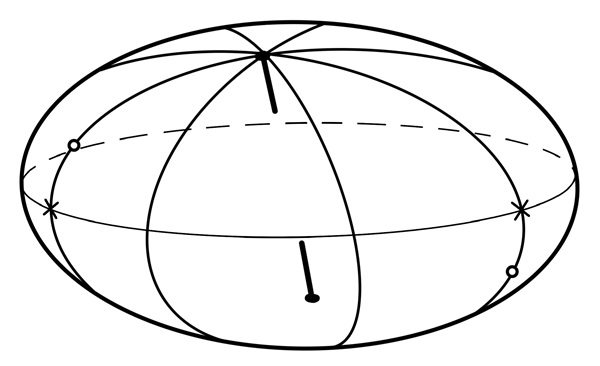}}
  \caption{\label{pic-lens-space-R-cut-locus} The cut locus for axisymmetric Riemannian metric on the lens space $L(p;q)$ for $p > 1$ and $\eta < -\frac{p-1}{p}$ has two strata $\partial L / \sim$ and an interval.}
\endminipage
\hfil
\minipage{0.45\textwidth}
  \center{\includegraphics[width=0.7\linewidth]{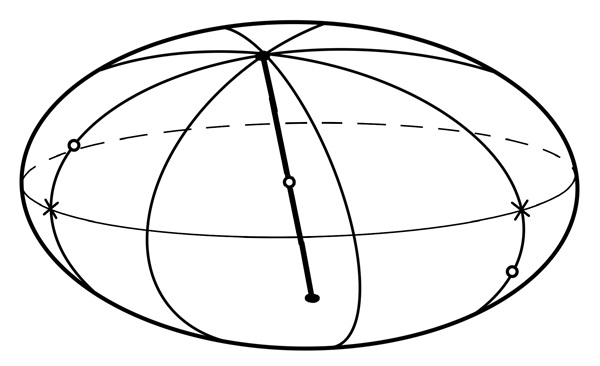}}
  \caption{\label{pic-lens-space-sR-cut-locus} The cut locus for axisymmetric sub-Riemannian metric on the lens space $L(p;q)$  for $p > 1$ has two strata $\partial L / \sim$ and a punctured circle. This corresponds to $\eta \rightarrow -1$.}
\endminipage}
\end{figure}

\begin{remark}
\label{rem-int}
Note that the additional stratum of the cut locus in item~(2b) of Theorem~\ref{th-cutlocus} is an interval since in the model of the lens space $L(p;q)$ described in Proposition~\ref{prop-lens-space} we should identify the points $\pm \sin{\frac{\pi}{p}}$ on the $q_3$-axis
(the North and the South poles of the ellipsoid of revolution).
\end{remark}

\begin{remark}
\label{rem-initial-point}
If $q = 1$, then the lens space $L(p;1)$ is homogeneous and the structure of the cut time and the cut locus doesn't depend on the initial point.
For $q > 1$ it is significant that we consider an initial point $o = \Pi(\id)$.
For the topological structure of the two dimensional component of the cut locus we refer to the work of S.~Anisov~\cite{anisov} were the cut locus for a symmetric
($I_1 = I_3$) Riemannian metric is studied. Note that the cut locus for a symmetric metric for $p > 1$ coincides with two dimensional component $\partial L / \sim$ of the cut locus in our case.
In paper~\cite{anisov}, the cut locus for an arbitrary initial point for symmetric metric is also described.
\end{remark}

\begin{remark}
\label{rem-acc}
Theorem~\ref{th-cutlocus} agrees with the results obtained in~\cite[Th.~5, Th.~3]{podobryaev-sachkov-so3} for $L(1;1) = \SU_2$ and $L(2;1) = \SO_3$.
\end{remark}

\begin{remark}
\label{rem-p1-cut-locus}
If $p = 1$, then the surface $\ell(q) = q_3^2 = 0$ is a two dimensional sphere $S^2 = \{q_0^2 + q_1^2 + q_2^2 = 1\}$, see Remark~\ref{rem-p1}.
In this case for $\eta > 0$ not the whole sphere is a Maxwell set and consequently the cut locus,
but only a two dimensional disk on this sphere. The boundary of this disk consists of conjugate points.
This is the result of T.~Sakai~\cite{sakai} for the cut locus of the Berger sphere in the case $\eta > 0$.
The reason of this phenomenon is that for a point on this sphere outside the disk there is only one geodesic coming to this point in time less than the cut time,
since this Maxwell set (the disk) corresponds to the reflection with respect to the plane $h_3 = 0$ and this symmetry has stationary points on this plane $h_3 = 0$
in contrast with $\Z_p$-symmetries providing the Maxwell strata for $p > 1$.
If $\eta = 0$, then this disk becomes a point.
If $-1 < \eta < 0$, then the cut locus is an interval at the $q_3$-axis as in Theorem~\ref{th-cutlocus}\,(2b), see~\cite[Th.~5]{podobryaev-sachkov-so3} for details.
\end{remark}

\begin{proof}[Proof of Theorem~\emph{\ref{th-cutlocus}}]
First, the expression for $\taul$ follows from Lemma~\ref{lem-taulminus-taulplus}.

Second, let us find the first Maxwell points corresponding to rotations, i.e., the first Maxwell time with
$\tau = \pi$ for $\eta < -\frac{p-1}{p}$ and $|\bar{h}_3| \geqslant \frac{p-1}{p|\eta|}$.
Substituting these values to geodesic equations of Proposition~\ref{prop-geodesics} we get the segment~(2b).
While the first Maxwell point corresponding to the first Maxwell time defined by $\taul$ gives the two dimensional component $\partial L / \sim$.
So, we have described the set of the first Maxwell points $M_o^{\Sym}$.

Now it is sufficient to prove that the exponential map is a diffeomorphism
$$
\Exp : U = \{ (h,t) \in C\times\R_+ \, | \, 0 < t < \tmax^{\Sym}(h) \} \rightarrow L(p;q) \setminus \left( \closure{M_o^{\Sym}} \cup \{o \} \right).
$$
We use the Hadamard global diffeomorphism theorem~\cite{kranz-parks}.
A smooth non-degenerate proper map of two connected and simply connected manifolds of same dimensions is a diffeomorphism.
Indeed, these two domains are three dimensional connected and simply connected manifolds. Moreover, the exponential map is smooth and non-degenerate,
since the conjugate time is greater than or equal to the first Maxwell time due to Proposition~\ref{prop-maxwell-conj}.
It remains to show that our map is proper, i.e., for any compact $K \subset L(p;q) \setminus \left(\closure{M_o^{\Sym}} \cup \{o\} \right)$ the set $\Exp^{-1}{K}$ is compact as well.
It is sufficient to prove that $\Exp^{-1}{K}$ is closed.

Assume by contradiction that $\Exp^{-1}{K}$ is not closed. Then there exists a sequence $(h_n,t_n)$ that converges to some $(h,t) \in \closure{U} \setminus \Exp^{-1}{K}$.
Since $\Exp$ is a continuous map, then $\Exp{(h_n,t_n)}$ converges to $\Exp{(h,t)} \in K$, since $K$ is compact.

If $(h,t) \in U$, then $(h,t) \in \Exp^{-1}{K}$, we get a contradiction.

Consider now the case $(h,t) \in \partial U$. This means that $t = 0$ or $t$ equals the first Maxwell time $\tmax^{\Sym}(h)$.
Hence, $\Exp{(h,t)} \in \closure{M_o^{\Sym}} \cup \{o \}$, but $K$ is compact.
\end{proof}

\begin{theorem}
\label{th-subRiemannian}
The cut time and the cut locus for the Berger lens space $L(p;q)$ \emph{(}with respect to the initial point $o = \Pi(\id)$\emph{)} converge to the cut time and the cut locus of the axisymmetric sub-Riemannian structure on $L(p;q)$ when $\eta \rightarrow -1$ \emph{(}or, equivalent, $I_3 \rightarrow +\infty$\emph{)},
see Fig.~\emph{\ref{pic-lens-space-sR-cut-locus}}.
\end{theorem}

\begin{proof}
It is a direct computation.
We refer to paper~\cite[Th.~4]{boscain-rossi} for the cut time and the cut locus of the sub-Riemannian structure.
Note that the segment~(2b) in Theorem~\ref{th-cutlocus} is closing up to a punctured circle while $\eta \rightarrow -1$.
\end{proof}

\section{\label{sec-diam}Lower bound for diameter}

In this Section, we find a lower bound for diameter of an axisymmetric Riemannian metric on a lens space. More precisely, we find the maximum of distances from the point $o$ to other points. Obviously, this maximum equals to the maximum of the cut time. Moreover, for lens spaces $L(p;1)$ this is the exact value of the diameter, since these spaces are homogeneous. So, we need to study the cut time as a function of variable $\bar{h}_3$.
We need several technical lemmas.

\begin{lemma}
\label{lem-even}
The cut time $\tcut : [0,1] \rightarrow \R_+$ is an even function of variable $\bar{h}_3$.
\end{lemma}

\begin{proof}
It follows from formulas~\eqref{eq-ell} that the function $\ell = \ell_- \ell_+$ is an even function of the variable $\bar{h}_3$.
Thus, its first positive root $\taul$ is even too (see also Lemma~\ref{lem-taulminus-taulplus}).
For $|h|$ as a function of the variable $\bar{h}_3$ we have
\begin{equation}
\label{eq-|h|}
\frac{h_1^2 + h_2^2}{I_1} + \frac{h_3^2}{I_3} = 1 \qquad \Rightarrow \qquad |h| = \frac{\sqrt{I_1}}{\sqrt{1+\eta\bar{h}_3^2}}.
\end{equation}
Hence, $|h|$ is even, and using the expression of the cut time from Theorem~\ref{th-cutlocus}~(1) we obtain that the cut time is an even function as well.
\end{proof}

Now we assume that $\bar{h}_3 \in [0,1]$ without loss of generality.
We see from Theorem~\ref{th-cutlocus}\,(1) that the cut time depends on the two functions $f$ and $g$, where
$$
\begin{array}{lcl}
f : (0, \frac{p-1}{p|\eta|}) \rightarrow \R, & \text{if} & -1 < \eta < -\frac{p-1}{p},\\
f : (0, 1) \rightarrow \R, & \text{if} & \eta \geqslant -\frac{p-1}{p},\\
\end{array}
\qquad
g : (0,1) \rightarrow \R,
$$
where, taking into account formula~\eqref{eq-|h|}, we have
\begin{equation}
\label{eq-fg}
f(\bar{h}_3) = \frac{2I_1\taul(\bar{h}_3)}{|h|} = 2\sqrt{I_1}\taul^-(\bar{h}_3)\sqrt{1+\eta\bar{h}_3^2}, \qquad
g(\bar{h}_3) = \frac{2I_1\pi}{|h|} = 2\pi\sqrt{I_1}\sqrt{1+\eta\bar{h}_3^2}.
\end{equation}
Below we find the critical points of these functions on the corresponding intervals.

\begin{lemma}
\label{lem-df}
If a point $\bar{h}_3 \in (0,1)$ is a critical point of the function $f$, then either
$$
\eta \neq 0 \qquad \text{and} \qquad \taul^-(\bar{h}_3) = \frac{\pi}{2} \qquad \text{and} \qquad \bar{h}_3 = \frac{(2k+1)p+2}{p\eta}, \qquad \text{where} \qquad k \in \Z,
$$
or $\taul^-(\bar{h}_3) = \tauconj(\bar{h}_3)$.
\end{lemma}

\begin{proof}
Let us compute the derivative of the function $f$, we get
$$
\frac{d f}{d \bar{h}_3} =
2\sqrt{I_1} \frac{d \taul^-}{d \bar{h}_3} \sqrt{1 +\eta\bar{h}_3^2} + 2\sqrt{I_1}\taul^-(\bar{h}_3)\frac{\eta\bar{h}_3}{\sqrt{1 + \eta\bar{h}_3^2}}.
$$
Using formula for $\frac{d \taul^-}{d \bar{h}_3}$ from Lemma~\ref{lem-dtau} we obtain (up to the positive multiplier $2\sqrt{I_1}/\sqrt{1 + \eta\bar{h}_3^2}$)
\begin{equation}
\label{eq-df-pgr1}
\frac{d f}{d \bar{h}_3} \sim
\frac{\cos{\left( \taul^-\eta\bar{h}_3 - \frac{\pi}{p} \right)} \left[-\taul^-\eta(1-\bar{h}_3^2)\cos{\taul^-} - (1 + \eta\bar{h}_3^2)\sin{\taul^-}\right]}{-(1+\eta\bar{h}_3^2)\sin{\taul^-}\sin{\left( \taul^-\eta\bar{h}_3 - \frac{\pi}{p} \right)} + \bar{h}_3(1+\eta)\cos{\taul^-}\cos{\left( \taul^-\eta\bar{h}_3 - \frac{\pi}{p} \right)}}.
\end{equation}
If $\cos{( \taul^-\eta\bar{h}_3 - \frac{\pi}{p} )} = 0$, then $\ell_-(\taul^-) = 0$ implies $\cos{\taul^-} = 0$.
Hence, if $\eta \neq 0$ we have
$$
\taul^- = \frac{\pi}{2} \qquad \text{and} \qquad \bar{h}_3 = \frac{(2k+1)p+2}{p\eta}, \qquad \text{where} \qquad  k \in \Z.
$$
Assume now that $\cos{( \taul^-\eta\bar{h}_3 - \frac{\pi}{p} )} \neq 0$ and $\bar{h}_3 \neq 0$.
It follows from the definition of $\taul^-$ that $\cos{\taul^-} \neq 0$.
Dividing the nominator and the denominator by $(1 + \eta\bar{h}_3^2) \cos{\taul^-} \cos{( \taul^-\eta\bar{h}_3 - \frac{\pi}{p} )}$ we get
$$
\frac{d f}{d \bar{h}_3} \sim
-\frac{\taul^-\eta\frac{1-\bar{h}_3^2}{1 + \eta\bar{h}_3^2} + \tan{\taul^-}}{-\tan{\taul^-}\tan{\left( \taul^-\eta\bar{h}_3 - \frac{\pi}{p} \right)} + \bar{h}_3(1+\eta)}.
$$
Moreover, by definition of $\taul^-$ we know that
$\tan{\left(\taul^-\eta\bar{h}_3 - \frac{\pi}{p}\right)} + \bar{h_3}\tan{\taul^-} = 0$.
This implies
\begin{equation}
\label{eq-df-p1}
\frac{d f}{d \bar{h}_3} \sim
-\frac{\taul^-\eta\frac{1-\bar{h}_3^2}{1 + \eta\bar{h}_3^2} + \tan{\taul^-}}{\bar{h}_3\tan^2{\taul^-} + \bar{h}_3(1+\eta)}.
\end{equation}
Notice that the expression in the nominator equals to the expression in the equation for the conjugate time~\eqref{eq-conj}.
So, if the derivative of the function $f$ vanishes at a point $\bar{h}_3$, then $\taul^-(\bar{h}_3) = \tauconj(\bar{h}_3)$.
\end{proof}

\begin{lemma}
\label{lem-f-critical}
Let us consider the function $f$ defined on the interval $(0, \frac{p-1}{p|\eta|})$ if $-1 < \eta < -\frac{p-1}{p}$
and on the interval $(0,1)$ if $\eta \geqslant -\frac{p-1}{p}$.\\
\emph{(1)} If $\eta < 0$, then there is at most one critical point of the function $f$ on the domain of its definition.
Moreover, if it exists, then it is not a point of maximum.\\
\emph{(2)} If $\eta = 0$, then there are no critical points or the function $f$ is constant.\\
\emph{(3)} If $\eta > 0$, then if $p = 1$ and $\eta > 1$ there exists the only one critical point $\frac{1}{\eta}$ of the function $f$ on the domain of its definition.
In other case there are no critical points. Moreover, if $p = 1$, then $\lim\limits_{\bar{h}_3 \rightarrow 0+}{\frac{df}{d\bar{h}_3}} > 0$.
\end{lemma}

\begin{proof}
(1) Due to Lemma~\ref{lem-df} and since by Proposition~\ref{prop-maxwell-conj} the inequality $\taul^- < \tauconj$ is satisfied on the intervals under consideration,
a necessary condition for critical point is
$$
\frac{(2k+1)p+2}{p\eta} \in (0,1) \qquad \Leftrightarrow \qquad
\frac{p(\eta-1)-2}{2p} < k < -\frac{p+2}{2p}.
$$
First, note that
$-2 < \frac{p(\eta-1)-2}{2p} < -\frac{p+2}{2p}$,
because these inequalities are equivalent to the following ones (respectively):
$$
-3 + \frac{2}{p} \leqslant -1 < \eta, \qquad p\eta < 0.
$$
Second,
$$
-1 < -\frac{p+2}{2p} < 0 \qquad \Leftrightarrow \qquad p > 2,
$$
and in this case the only one number $k = -1$ may satisfy the requiring inequalities.
This means that there is at most one critical point for $p > 2$ and for $p \leqslant 2$ there are no critical points.

We claim that it is not a point of maximum.
Indeed, by formula~\eqref{eq-df-pgr1} for $p > 1$ we have $\taul^-(0) = \frac{\pi}{2}$ and
if $\bar{h}_3 \rightarrow 0+$, then $\frac{df}{d\bar{h}_3} \rightarrow \frac{\cos{(-\frac{\pi}{p})}}{\sin{(-\frac{\pi}{p})}}$
and $\frac{\cos{(-\frac{\pi}{p})}}{\sin{(-\frac{\pi}{p})}} < 0$ for $p > 2$.

(2) If $\eta = 0$, then $|h|$ is constant. Next, for $p = 1$ or $p = 2$ the function $\taul^-$ is constant and the function $f$ is constant as well.
If $p > 2$, then it follows from the definition of $\taul^-$ that $\taul^-(\bar{h}_3) = \arctan{(\frac{1}{\bar{h}_3}\tan{\frac{\pi}{p}})}$.
This function has no critical points.

(3) Assume that $\eta > 0$. By Proposition~\ref{prop-maxwell-conj} we have $\taul^- < \tauconj$ on the intervals under consideration.
Thanks to Lemma~\ref{lem-df} we know a necessary condition for critical point, namely, $\taul^-(\bar{h}_3) = \frac{\pi}{2}$.
Also we know from Lemma~\ref{lem-decrease-special} that for $\eta > 0$ if $\taul^-(\hat{h}_3) = \frac{\pi}{2}$, then $\frac{d\taul^-}{d\bar{h}_3}(\hat{h}_3) < 0$.
This means that there is at most one critical point.

If $p > 1$, then $\taul^-(0) = \frac{\pi}{2}$. Since the function $\taul^-$ decreases by Lemma~\ref{lem-decrease-1}, then it does not take value $\frac{\pi}{2}$ on the interval $(0,1)$.
So, there are no critical points of the function $f$.
If $p = 1$ and $\eta > 1$, then there is the only one critical point $\frac{1}{\eta}$, it corresponds to $k = -1$.

Next, assume that $p = 1$ and $\eta > 0$, it follows from Lemma~\ref{lem-p1-limit} that $\taul^-(\bar{h}_3) \rightarrow \tauconj(0)-$ when $\bar{h}_3 \rightarrow 0+$.
By Proposition~\ref{prop-conjugate-time}\,(2) $\frac{\pi}{2} < \tauconj(0) \leqslant \pi$,
thus, $\frac{\pi}{2} < \taul^-(\bar{h}_3) < \tauconj(0) < \tauconj(\bar{h}_3) \leqslant \pi$ for $\bar{h}_3$ small enough.
This implies that the nominator of the formula~\eqref{eq-df-p1} is negative, while the denominator is positive.
Hence, $\lim\limits_{\bar{h}_3 \rightarrow 0+}{\frac{df}{d\bar{h}_3}} > 0$.
\end{proof}

\begin{theorem}
\label{th-diam}
\emph{(1)} There are the following lower bounds for diameter of the Berger lens space $L(p;q)$ where $\eta = \frac{I_1}{I_3} - 1$.
\begin{enumerate}
\item[\emph{(a)}] If $-1 < \eta < -\frac{p-1}{p}$, then
$$
\diam{L(p;q)} \geqslant \left\{
\begin{array}{rll}
2\pi\sqrt{I_1}\sqrt{1+\frac{(p-1)^2}{p^2\eta}}, & \text{if} & p \leqslant 4 \text{ or } 4 < p < 8 \text{ and } \eta < -\frac{4(p-1)^2}{3p^2},\\
\pi\sqrt{I_1}, & \text{else}. & \\
\end{array}
\right.
$$
\item[\emph{(b)}] If $-\frac{p-1}{p} \leqslant \eta < 0$, then
$$
\diam{L(p;q)} \geqslant \left\{
\begin{array}{rll}
\pi\sqrt{I_1}, & \text{if} & p = 3 \text{ and } \eta \geqslant -\frac{5}{9} \text{ or } p \geqslant 4,\\
\frac{2\pi\sqrt{I_3}}{p}, & \text{else}. & \\
\end{array}
\right.
$$
\item[\emph{(c)}] If $\eta = 0$, then
$$
\diam{L(p;q)} \geqslant \left\{
\begin{array}{rcl}
2\pi\sqrt{I_1}, & \text{if} & p = 1,\\
\pi\sqrt{I_1}, & \text{if} & p \geqslant 2.\\
\end{array}
\right.
$$
\item[\emph{(d)}] If $\eta > 0$, then
$$
\diam{L(p;q)} \geqslant \left\{
\begin{array}{rll}
\pi\sqrt{I_1}, & \text{if} & p > 1,\\
2\pi\sqrt{I_3}, & \text{if} & p = 1 \text{ and } 0 < \eta \leqslant 1,\\
\frac{\pi I_1}{\sqrt{I_1 - I_3}}, & \text{if} & p = 1 \text{ and } \eta > 1.\\
\end{array}
\right.
$$
\end{enumerate}
\emph{(2)} For the Berger lens space $L(p;1)$ these bounds are exact values of diameters.
\end{theorem}

\begin{remark}
\label{rem-diam-metric}
In the case of a symmetric metric (i.e., $\eta = 0$), usually the standard metric $s$ on the unit sphere $S^3 \subset \R^4$ and the corresponding metric on the lens space $L(p;q)$ are considered.
Since our metric differs by a multiplier (see Remark~\ref{rem-metric}), then diameter of our metric is $2\sqrt{I_1}$ times greater than diameter of the standard metric.
Thus, our result of Theorem~\ref{th-diam}~(1c) agrees with~\cite[Lemma~2.1]{anisov}.
\end{remark}

\begin{remark}
\label{rem-prev-diam}
These values of diameter for $p=1$ and $p=2$ coincide with the previous results, see~\cite[Th.~1]{podobryaev-berger-sphere-diameter} and \cite[Th.~4]{podobryaev-sachkov-so3}, respectively.
\end{remark}

\begin{remark}
\label{rem-diam-cont-eta}
These diameter bounds are continuous as a functions of the variable $\eta$ for fixed $p$.
\end{remark}

\begin{proof}[Proof of Theorem~\emph{\ref{th-diam}}]
Let us illustrate the proof by Fig.~\ref{pic-tcut-peaks} which presents the cases when the cut time has a maximum point that is different from the ends of the segment $[0,1]$.

\begin{figure}
\centering{
\minipage{0.50\textwidth}
  \center{\includegraphics[width=0.55\linewidth]{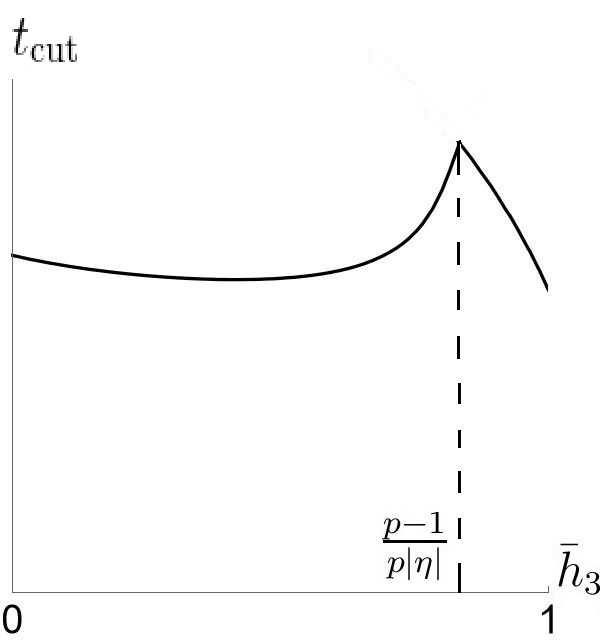} \\ (a) $\eta < -\frac{p-1}{p}, p \leqslant 4 \text{ or } 4 < p < 8, \eta < -\frac{4(p-1)^2}{3p^2}$}
\endminipage
\hfil
\minipage{0.50\textwidth}
  \center{\includegraphics[width=0.55\linewidth]{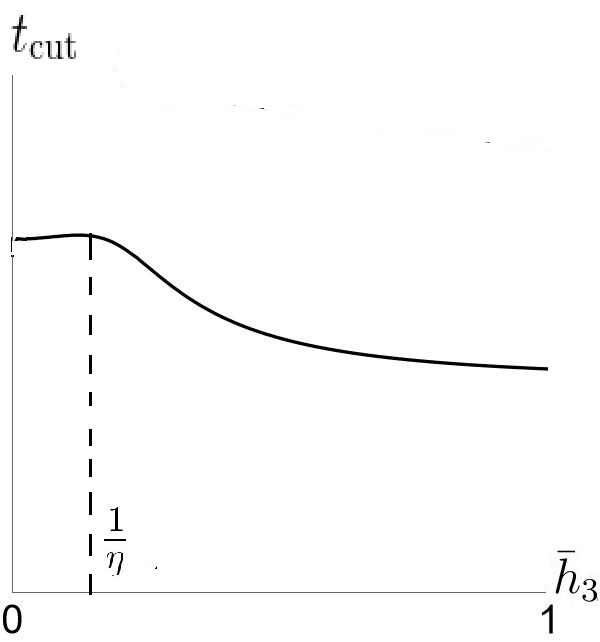} \\ (b) $p = 1, \eta > 1$}
\endminipage}
\caption{\label{pic-tcut-peaks} The cases where the cut time $\tcut: [0,1] \rightarrow \R_+$ has maximum that differs from $\tcut(0)$ and $\tcut(\pm 1)$.}
\end{figure}

(1a) In this case the cut time as a function of variable $\bar{h}_3$ is piecewise smooth, see Theorem~\ref{th-cutlocus}~(1).
Moreover, on the segment $\left[\frac{p-1}{p|\eta|}, 1\right]$ the function $g$ decreases, see formula~\eqref{eq-fg}.
Due to Lemma~\ref{lem-f-critical}~(1) the function $f$ has at most one critical point on the interval $(0, \frac{p-1}{p|\eta|})$ and it is not a point of maximum.
Hence, we need to compare to values: $f(0)$ and $f(\frac{p-1}{p|\eta|}) = g(\frac{p-1}{p|\eta|})$ where
$$
g \left(\frac{p-1}{p|\eta|}\right) = 2\pi\sqrt{I_1}\sqrt{1 + \frac{(p-1)^2}{p^2\eta}}.
$$

First, consider the case $p = 1$. In this case, the cut time is defined by the function $g$ and $g(0) = 2\pi\sqrt{I_1} > 2\pi\sqrt{I_1}\sqrt{1+\eta} = g(1)$.

Second, consider the case $p > 1$. Since $\taul^-(0) = \frac{\pi}{2}$ we obtain $f(0) = 2\taul^-(0)\sqrt{I_1} = \pi\sqrt{I_1}$ and
$$
f(0) \leqslant g \left(\frac{p-1}{p|\eta|}\right) \qquad \Leftrightarrow \qquad (3\eta + 4)p^2 - 8p + 4 \leqslant 0 \qquad \Leftrightarrow \qquad
\eta \leqslant -\frac{4(p-1)^2}{3p^2}.
$$
Notice that $-\frac{4(p-1)^2}{3p^2} \geqslant -\frac{p-1}{p}$ iff $p \leqslant 4$.
Moreover,
$$
-\frac{4(p-1)^2}{3p^2} > -1 \qquad \Leftrightarrow \qquad p^2 - 8p + 4 < 0 \qquad \Leftrightarrow \qquad p \in (4-2\sqrt{3}, 4 + 2\sqrt{3}).
$$
Since $0 < 4 - 2\sqrt{3} < 1$ and $7 < 4 + 2\sqrt{3} < 8$, then $p < 8$.

(1b) In this case the cut time is equal to the function $f$, see Theorem~\ref{th-cutlocus}\,(1).
Since by Lemma~\ref{lem-f-critical}\,(1) there is at most one critical point of the function $f$ and it is not a point of maximum,
we need to compare the values of the function $f$ at the ends of the interval $[0,1]$.
Note that in this case $p \neq 1$. We obtain
\begin{equation}
\label{eq-f01}
f(0) = 2\taul^-(0)\sqrt{I_1} = \pi\sqrt{I_1}, \qquad f(1) = 2\taul^-(1)\sqrt{I_1}\sqrt{1+\eta} = \frac{2\pi\sqrt{I_1}}{p\sqrt{1+\eta}} = \frac{2\pi\sqrt{I_3}}{p},
\end{equation}
since $\taul^-(1) = \frac{\pi}{p(1+\eta)}$.
We have the following equivalent inequalities:
\begin{equation}
\label{eq-f01-comparsion}
f(0) \geqslant f(1) \qquad \Leftrightarrow \qquad 1 \geqslant \frac{2}{p\sqrt{1+\eta}} \qquad \Leftrightarrow \qquad \eta \geqslant \frac{4}{p^2} - 1.
\end{equation}
Note that $\frac{4}{p^2} - 1 \leqslant -\frac{p-1}{p}$ iff $p \geqslant 4$ and this condition satisfies automatically.
It remains to consider  $1 < p < 4$. If $p = 2$, then $\eta \geqslant \frac{4}{p^2} - 1 = 0$, but $\eta < 0$ in our case.
If $p = 3$, then $\eta \geqslant \frac{4}{p^2} - 1 = -\frac{5}{9}$.

(1c) Immediately follows from Lemma~\ref{lem-f-critical}~(2). Indeed, for $p = 1$ the function $f$ is constant and equals $2\pi\sqrt{I_1}$, since $\taul^- = \pi$.
If $p \geqslant 2$, then $f(0) = \pi\sqrt{I_1} \geqslant \frac{2\pi\sqrt{I_1}}{p} = f(1)$, see~\eqref{eq-f01}.

(1d) Assume that $p > 1$. The cut time is equal to the function $f$ by Theorem~\ref{th-cutlocus}\,(1) and this function has no critical points by Lemma~\ref{lem-f-critical}\,(3).
We already know~\eqref{eq-f01-comparsion} that $f(0) \geqslant f(1)$ iff $\eta \geqslant \frac{4}{p^2} - 1$.
But for $p \geqslant 2$ we obtain $\frac{4}{p^2} - 1 \leqslant 0$ and the condition $f(0) \geqslant f(1)$ holds automatically in our case $\eta > 0$.

Assume now that $p = 1$. It follows from Lemma~\ref{lem-f-critical}\,(3) that if $0 < \eta \leqslant 1$ the maximum value of the function $f$ is
$$
f(1) = 2\taul^-(1)\sqrt{I_1}\sqrt{1+\eta} = \frac{2\pi\sqrt{I_1}}{\sqrt{1+\eta}} = 2\pi\sqrt{I_3},
$$
since $\taul^-(1) = \frac{\pi}{1+\eta}$.
If $\eta > 1$, then we need to compare $f(\frac{1}{\eta})$ and $f(1)$. Since $\taul^-(\frac{1}{\eta}) = \frac{\pi}{2}$ we obtain
$$
f \left( \frac{1}{\eta} \right) = 2\taul^- \left( \frac{1}{\eta} \right) \sqrt{1 + \frac{1}{\eta}} = \pi \sqrt{I_1} \sqrt{1 + \frac{1}{\eta}} \geqslant f(1)
\qquad \Leftrightarrow \qquad (\eta - 1)^2 \geqslant 0.
$$
So, for any $\eta > 1$ we have $f(\frac{1}{\eta}) \geqslant f(1)$.
It remains to note that $\pi \sqrt{I_1} \sqrt{1 + \frac{1}{\eta}} = \frac{\pi I_1}{I_1 - I_3}$.

(2) This is due to the fact that the lens space $L(p;1)$ is homogeneous.
\end{proof}

\section*{Acknowledgment}

The author is deeply grateful to the anonymous reviewer for carefully reading the manuscript and pointing out an error in the original version of Theorem~\ref{th-diam} and a lot of typos.


\begin{thebibliography}{99}

\bibitem{podobryaev-sachkov-so3}
Podobryaev, A.\,V. and Sachkov, Yu.\,L., Cut locus of a left invariant Riemannian metric on $SO(3)$ in the axisymmetric case, \textit{Journal of Geometry and Physics}, 2016, vol.\,110, pp.\,436--453.

\bibitem{podobryaev-berger-sphere-diameter}
Podobryaev, A.\,V., Diameter of the Berger Sphere, \textit{Mathematical Notes}, 2018, vol.\,103, no.\,5, pp.\,846--851.

\bibitem{podobryaev-sachkov-sl2}
Podobryaev, A.\,V. and Sachkov, Yu.\,L., Symmetric Riemannian problem on the group of proper isometries of hyperbolic plane, \textit{Journal of Dynamical and Control Systems}, 2018, vol.\,24, no.\,3, pp.\,391--423.

\bibitem{podobryaev-sachkov-so3-sl2}
Podobryaev, A.\,V. and Sachkov, Yu.\,L., Left-invariant Riemannian problems on the groups of proper motions of hyperbolic plane and sphere, \textit{Doklady Mathematics}, 2017, vol.\,95, no.\,2, pp.\,176--177.

\bibitem{berestovskii-zubareva-sl2}
Berestovskii, V.\,N. and Zubareva, I.\,A., Geodesics and shortest arcs of a special sub-Riemannian metric on the Lie group $SL(2)$, \textit{Siberian Math. J.}, 2016, vol.\,57, no.\,3, pp.\,411--424.

\bibitem{berestovskii-zubareva-so3}
Berestovskii, V.\,N. and Zubareva, I.\,A., Geodesics and shortest arcs of a special sub-Riemannian metric on the Lie group $SO(3)$, \textit{Siberian Math. J.}, 2015, vol.\,56, no.\,4, pp.\,601--611.

\bibitem{berestovskii-zubareva-su2}
Berestovskii, V.\,N. and Zubareva, I.\,A., Sub-Riemannian distance in the Lie groups $SU(2)$ and $SO(3)$, \textit{Siberian Adv. Math.}, 2016, vol.\,26, no.\,2, pp.\,77--89.

\bibitem{boscain-rossi}
Boscain, U. and Rossi, F., Invariant Carnot-Caratheodory metrics on $S^3$, $\operatorname{SO}(3)$, $\operatorname{SL}(2)$ and lens spaces, \textit{SIAM Journal on Control and Optimization}, 2008, vol.\,47, pp.\,1851--1878.

\bibitem{anisov}
Anisov, S., Cut loci in lens manifolds, \textit{C. R. Acad. Sci. Paris, Ser. I.}, 2006, vol.\,342, pp.\,595--600.

\bibitem{duits-pse2}
Bekkers, E.\,J., Duits, R., Mashtakov, A., and Sachkov, Yu., Vessel tracking via sub-Riemannian geodesics on the projective line bundle,
in \textit{Geometric Science of Information. GSI 2017}, F.\,Nielsen, F.\,Barbaresco (Eds.),
Lecture Notes in Computer Science, Springer, Cham, 2017, vol.\,10589, pp.\,773--781.

\bibitem{duits-so3}
Mashtakov, A., Duits, R., Sachkov, Yu., Bekkers, E.\,J., and Beschastnyi, I., Tracking of lines in spherical images via sub-Riemannian geodesics in SO(3),
\textit{Journal of Mathematical Imaging and Vision}, 2017, vol.\,58, no.\,2, pp.\,239--264.

\bibitem{bates-fasso}
Bates, L. and Fass\`{o}, F., The conjugate locus for the Euler top. I. The axisymmetric case, \textit{Int. Math. Forum}, 2007, vol.\,2, no.\,43, pp.\,2109--2139.

\bibitem{kowalski-vanhecke}
Kowalski, O. and Vanhecke, L., Riemannian manifolds with homogeneous geodesics, \textit{Boll. Unione Mat. Ital. Ser. B}, 1991, vol.\,5, no.\,1, pp.\,189--246.

\bibitem{kranz-parks}
Krantz, S.\,G. and Parks, H.\,R., \textit{The Implicit Function Theorem: History, Theory and Applications}, Birkauser, 2003.

\bibitem{sakai}
Sakai, T., Cut loci of Berger's sphere, \textit{Hokkaido Math. J.}, 1981, vol/\,10, pp.\,143--155.

\end{thebibliography}
\end{document}